\documentclass[reqno,a4paper]{amsart}
\pdfoutput=1
\usepackage{fullpage,amssymb}

\usepackage{enumitem}

\usepackage[nobysame,alphabetic,initials,msc-links]{amsrefs}

\DefineSimpleKey{bib}{how}
\renewcommand{\PrintDOI}[1]{%
  \href{http://dx.doi.org/#1}{{\tt DOI:#1}}%
}
\renewcommand{\eprint}[1]{#1}
\BibSpec{book}{%
    +{}  {\PrintPrimary}                {transition}
    +{.} { \PrintDate}                  {date}
    +{.} { \textit}                     {title}
    +{.} { }                            {part}
    +{:} { \textit}                     {subtitle}
    +{,} { \PrintEdition}               {edition}
    +{}  { \PrintEditorsB}              {editor}
    +{,} { \PrintTranslatorsC}          {translator}
    +{,} { \PrintContributions}         {contribution}
    +{,} { }                            {series}
    +{,} { \voltext}                    {volume}
    +{,} { }                            {publisher}
    +{,} { }                            {organization}
    +{,} { }                            {address}
    +{,} { }                            {status}
    +{,} { \PrintDOI}                   {doi}
    +{,} { \PrintISBNs}                 {isbn}
    +{}  { \parenthesize}               {language}
    +{}  { \PrintTranslation}           {translation}
    +{;} { \PrintReprint}               {reprint}
    +{.} { }                            {note}
    +{.} {}                             {transition}
    +{}  {\SentenceSpace \PrintReviews} {review}
}
\BibSpec{article}{%
    +{}  {\PrintAuthors}                {author}
    +{,} { \textit}                     {title}
    +{.} { }                            {part}
    +{:} { \textit}                     {subtitle}
    +{,} { \PrintContributions}         {contribution}
    +{.} { \PrintPartials}              {partial}
    +{,} { }                            {journal}
    +{}  { \textbf}                     {volume}
    +{}  { \PrintDatePV}                {date}
    +{,} { \issuetext}                  {number}
    +{,} { \eprintpages}                {pages}
    +{,} { }                            {status}
    +{,} { \PrintDOI}                   {doi}
    +{,} { \eprint}        {eprint}
    +{}  { \parenthesize}               {language}
    +{}  { \PrintTranslation}           {translation}
    +{;} { \PrintReprint}               {reprint}
    +{.} { }                            {note}
    +{.} {}                             {transition}
    +{}  {\SentenceSpace \PrintReviews} {review}
}
\BibSpec{collection.article}{%
    +{}  {\PrintAuthors}                {author}
    +{,} { \textit}                     {title}
    +{.} { }                            {part}
    +{:} { \textit}                     {subtitle}
    +{,} { \PrintContributions}         {contribution}
    +{,} { \PrintConference}            {conference}
    +{}  {\PrintBook}                   {book}
    +{,} { }                            {booktitle}
    +{,} { \PrintDateB}                 {date}
    +{,} { pp.~}                        {pages}
    +{,} { }                            {publisher}
    +{,} { }                            {organization}
    +{,} { }                            {address}
    +{,} { }                            {status}
    +{,} { \PrintDOI}                   {doi}
    +{,} { \eprint}        {eprint}
    +{}  { \parenthesize}               {language}
    +{}  { \PrintTranslation}           {translation}
    +{;} { \PrintReprint}               {reprint}
    +{.} { }                            {note}
    +{.} {}                             {transition}
    +{}  {\SentenceSpace \PrintReviews} {review}
}
\BibSpec{misc}{%
  +{}{\PrintAuthors}  {author}
  +{,}{ \textit}      {title}
  +{.}{ }             {how}
  +{}{ \parenthesize} {date}
  +{,} { available at \eprint}        {eprint}
  +{,}{ available at \url}{url}
  +{,}{ }             {note}
  +{.}{}              {transition}
}

\numberwithin{equation}{section}

\newtheorem{theorem}{Theorem}[section]

\newtheorem{lemma}[theorem]{Lemma}
\newtheorem{proposition}[theorem]{Proposition}

\theoremstyle{remark}
\newtheorem{remark}[theorem]{Remark}
\newtheorem{example}[theorem]{Example}

\theoremstyle{definition}
\newtheorem*{convention}{Convention}
\newcommand{\bC}{\mathbb{C}}
\newcommand{\bR}{\mathbb{R}}

\newcommand{\amn}{\mathbin{\triangledown}}
\newcommand{\gmn}{\mathbin{\#}}
\newcommand{\hmn}{\mathbin{!}}
\newcommand{\smn}{\mathbin{\sigma}}
\newcommand{\hlf}{1/2}%{\sfrac12}
\newcommand{\absv}[1]{\left\lvert#1\right\rvert}
\DeclareMathOperator{\Tr}{Tr}
\newcommand{\sa}{\mathrm{sa}}

\title[Strengthened convexity of operator decreasing functions]{Strengthened convexity of positive operator monotone decreasing functions}

\author{Megumi Kirihata}
\address{Department of Mathematics, Ochanomizu University, Otsuka 2-1-1, Bunkyo-ku, Tokyo  112-8610, Japan}
\email{krmg.94@ezweb.ne.jp}
\author{Makoto Yamashita}
\address{Department of Mathematics, University of Oslo, P.O box 1053, Blindern, 0316 Oslo, Norway}
\email{makotoy@math.uio.no}

\date{v3: May 19, 2021 (post-publication changes); v2: September 12, 2019; v1: March 3, 2019}

\begin{document}

\begin{abstract}
We prove a strengthened form of convexity for operator monotone decreasing positive functions defined on the positive real numbers.
This extends Ando and Hiai's work to allow arbitrary positive maps instead of states (or the identity map), and functional calculus by operator monotone functions defined on the positive real numbers instead of the logarithmic function.
\end{abstract}

\maketitle

\section{Introduction}

The theory of operator monotone and convex functions initiated by L\"{o}wner and Kraus, and modernized by Choi~\cite{MR0355615}, Ando~\citelist{\cite{ando-lect-note-op-ineq}\cite{MR535686}}, and Hansen and Pedersen~\cite{MR649196} in connection to positive linear maps of operators, reveals interesting relations between function theoretic concepts on the one hand, and structure of positive or self-adjoint operators on Hilbert spaces on the other.
Compatibility with the (partial) order relation of such operators forces strong regularity on functions, and in particular the operator monotone decreasing functions $f(x)$ on an interval admit, up to linear terms, integral representations with $1/(\lambda + x)$ as integrand.
These functions are \emph{logarithmically convex} (or superconvex) besides being monotone decreasing, which suggests that operator monotone (decreasing) functions automatically have stronger form of concavity / convexity.

The corresponding notion of operator logarithmic convexity was first considered in~\cite{MR1845701}, and in an interesting paper~\cite{MR2805638}, Ando and Hiai showed that operator monotone decreasing positive functions $f(x)$ defined on positive real numbers indeed admit operator log convexity.
Moreover, they showed that composition of states and such functions have logarithmic convexity, that is, if $\omega$ is a state, the map $X \mapsto \log \omega(f(X))$ is convex for positive invertible operators $X$.
In another direction, Kian and Dragomir \cite{MR3434508} gave a characterization of operator log convexity by a strengthened form of the Jensen inequality.

In this short note we show that the functions in this class have even stronger form of convexity, by allowing $\log(x)$ and $\omega(X)$ above to be of more general forms.
Our main result (Theorem~\ref{thm:main}) states that, if $g(x)$ is an operator monotone function defined for $0 < x < \infty$, and if $\Phi$ is a strictly positive linear map of operators, then the map $X \mapsto g(\Phi(f(X)))$ is convex on invertible positive operators.

This generalization is comparable to Hiai's more recent work~\citelist{\cite{MR3067823}\cite{MR3464069}}, in which he considers the joint concavity / convexity problems for trace functionals of the form $\Tr(g(\Phi(X^p)^{\hlf}\Psi(Y^q)\Phi(X^p)^{\hlf}))$ with suitable $g(x)$ and positive maps $\Phi$ and $\Psi$, generalizing a foundational work of Lieb~\cite{MR0332080}.
See Section~\ref{sec:two-var-ver} for a more detailed comparison.

\medskip
\paragraph{Acknowledgements} We thank Fumio Hiai for encouragement and pointing us to~\citelist{\cite{MR3067823}\cite{MR3464069}}.
%It is our pleasure to thank the anonymous referee who pointed out an improvement of our main result from initial version using a result from
We would also like to thank the anonymous reviewer for drawing our attention to \cite{MR3434508}.

\medskip
\paragraph{Note added after publication}
May 19, 2021: there was an unfortunate typo in the formulation of Theorem \ref{thm:main}, where $B^{++}$ should have been $B_{\sa}$.
The proof is unchanged.
We also take this opportunity to note that the joint convexity of \eqref{eq:2-var-func} cannot be hoped for with $g(x) = - x^s$, $-1\le s < -1/2$ and $f_1(x) = f_2(x) = 1/x$ by the results of \cite{MR3067823}*{Section 5}.

\section{Preliminaries}

Let us fix our convention and review relevant basic facts.
See standard texts such as~\citelist{\cite{MR649196}\cite{MR1477662}} for the details.

\begin{convention}
In the following $A$ always denotes a unital C$^*$-algebra, such as $M_n(\bC)$ or $B(H)$ for some Hilbert space $H$.
\end{convention}

\subsection{Positive operators}

We denote the set of invertible positive elements of $A$ by $A^{++}$, and the set of selfadjoint elements by $A_{\sa}$.
When $B$ is another unital C$^*$-algebra, a linear map $\Phi\colon A \to B$ is said to be \emph{strictly positive} if it maps $A^{++}$ into $B^{++}$.

\subsection{Operator means}

When $X$ and $Y$ are elements of $A^{++}$, their \emph{harmonic mean} is
$$
X \hmn Y = \left(\frac{X^{-1} + Y^{-1}}{2}\right)^{-1}.
$$
The usual average
$$
X \amn Y = \frac{X + Y}{2}
$$
is called the \emph{arithmetic mean}, so that we can write $X \hmn Y = (X^{-1} \amn Y^{-1})^{-1}$.

\subsection{Operator monotone and convex functions}

Let $I$ be a subset of $\bR$.
A real function $f(x)$ on $I$ is said to be \emph{operator monotone} on $I$ if the functional calculus by $f$ satisfies
$$
X \le Y \Rightarrow f(X) \le f(Y)
$$
for all self-adjoint elements $X$, $Y$ in $A$ whose spectra $\sigma(X)$, $\sigma(Y)$ are contained in $I$.
If $X \le Y$ implies $f(X) \ge f(Y)$, we say that $f$ is \emph{operator monotone decreasing}.

Similarly, assuming $I$ to be an interval, $f$ is said to be \emph{operator convex} on $I$ if
$$
f(t X + (1-t) Y) \le t f(X) + (1-t) f(Y)
$$
holds for $X$, $Y$ as above and $0 \le t \le 1$.
\emph{Operator concavity} is defined by the reverse inequality.

If $f\colon (0, \infty) \to \bR$ is operator monotone, then it is operator concave. %, and the reverse implication holds if the range of $f$ is bounded from below.
Similarly, continuous operator monotone decreasing functions on $(0, \infty)$ are operator convex.

Let $\Phi\colon A \to B$ be a unital positive map between unital C$^*$-algebras, and $I = [a, b]$ be a closed interval.
If $X \in A$ is a self-adjoint element with $\sigma(X) \subset I$, the spectrum of $\Phi(X)$ is also in $I$.
Moreover, if $f(x)$ is an operator convex function on $I$, then we have the \emph{Jensen inequality}
$$
f(\Phi(X)) \le \Phi(f(X))
$$
for $X$ as above.

\section{Main result}

\begin{theorem}
\label{thm:main}
Let $f\colon (0, \infty) \to (0, \infty)$ be an operator monotone decreasing function, and $g\colon (0, \infty) \to \bR$ be an operator monotone function.
When $A$, $B$ are unital C$^*$-algebras and $\Phi \colon A \to B$ is a strictly positive linear map, the transform
$$
A^{++} \to B_{\sa},\quad X \mapsto g(\Phi(f(X)))
$$
is a convex map.
\end{theorem}

\begin{example}
\label{ex:main-thm-ex}
\begin{enumerate}[label=(\roman*)]
\item For $f(x) = x^{-1}$ and $g(x) = -x^{-1}$, the above reduces to the well-known concavity of the map $X \mapsto \Phi(X^{-1})^{-1}$.
\item Another important case is $g(x) = \log x$. For this $g(x)$, the cases of $\Phi(T) = T$ and $\Phi(T) = \omega(T)$ for some state $\omega$ were separately treated in~\cite{MR2805638}.
\end{enumerate}
\end{example}

We prove the above result through the following elementary lemmas.

\begin{lemma}
\label{lem:pos-map-harm-mean}
Let $\Phi \colon A \to B$ be a (strictly) positive linear map.
Then for (invertible) positive elements $X$ and $Y$ in $A$, we have
$$
\Phi(X \hmn Y) \le \Phi(X) \hmn \Phi(Y).
$$
\end{lemma}

\begin{proof}
This is essentially~\cite{MR2284176}*{Theorem~4.1.5 (i)}, or can be reduced to Example~\ref{ex:main-thm-ex} (i), but for the reader's convenience let us give a more direct argument:
combining $X \hmn Y = 2(X - X (X + Y)^{-1} X)$ and the linearity of $\Phi$, we can reduce the claim to
\begin{equation}
\label{eq:pos-map-harm-mean-lem-goal}
\Phi(X (X+Y)^{-1} X) \ge \Phi(X) \Phi(X+Y)^{-1} \Phi(X).
\end{equation}

Consider the unital positive linear map
$$
\Phi'_u(T) = \Phi(X)^{-\hlf} \Phi(X^{\hlf} T X^{\hlf}) \Phi(X)^{-\hlf}.
$$
Then the Jensen inequality $\Phi'_u(T^{-1}) \ge \Phi'_u(T)^{-1}$ applied to $T = X^{-\hlf} (X + Y) X^{-\hlf}$ implies that
$$
\Phi(X)^{-\hlf} \Phi(X (X+Y)^{-1} X) \Phi(X)^{-\hlf} \ge \Phi(X)^{\hlf} \Phi(X+Y)^{-1} \Phi(X)^{\hlf},
$$
which is equivalent to \eqref{eq:pos-map-harm-mean-lem-goal}.
\end{proof}

\begin{remark}
\label{rem:conv-sym-mean}
From the above lemma one can derive $\Phi(X \smn Y) \le \Phi(X) \smn \Phi(Y)$ for any symmetric operator mean $\sigma$ in the sense of~\cite{MR563399}.
\end{remark}

\begin{lemma}
\label{lem:op-mon-dec-harm-arith-mean}
Let $f(x)$ be as in the statement of Theorem~\ref{thm:main}.
Then for elements $X$ and $Y$ in $A^{++}$, we have
$$
f(X \amn Y) \le f(X) \hmn f(Y).
$$
\end{lemma}

\begin{proof}
This is observed in~\cite{MR2805638}*{p.~614}: $1/f(x)$ is operator monotone on $(0, \infty)$, hence is operator concave. The latter condition is equivalent to the above inequality.
\end{proof}

\begin{proof}[Proof of Theorem \ref{thm:main}]
By continuity, it is enough to prove the inequality
\begin{equation}
\label{goal-ineq}
g(\Phi(f(X \amn Y))) \le g(\Phi(f(X))) \amn g(\Phi(f(Y))).
\end{equation}

By assumption, the map
$$
A^{++} \to B_{\sa}, \quad T \mapsto g(\Phi(T))
$$
is monotone.
Combined with Lemma \ref{lem:op-mon-dec-harm-arith-mean}, we obtain
\begin{equation}
\label{eq:1}
g(\Phi(f(X \amn Y))) \le g(\Phi(f(X) \hmn f(Y))).
\end{equation}
By Lemma \ref{lem:pos-map-harm-mean} and the operator monotonicity of $g(x)$, we have
\begin{equation}
\label{eq:2}
g(\Phi(f(X) \hmn f(Y))) \le g(\Phi(f(X)) \hmn \Phi(f(Y))).
\end{equation}
Moreover $g(x^{-1})$ is operator monotone decreasing on $(0, \infty)$, hence it is operator convex, so
$$
g\mathopen{}\left((S \amn T)^{-1} \right) \le g(S^{-1}) \amn g(T^{-1})
$$
holds for $S$, $T$ in $B^{++}$.
Putting $S = \Phi(f(X))^{-1}$ and $T = \Phi(f(Y))^{-1}$, we obtain
\begin{equation}
\label{eq:3}
g(\Phi(f(X)) \hmn \Phi(f(Y))) \le g(\Phi(f(X))) \amn g(\Phi(f(Y))).
\end{equation}

Collecting the inequalities \eqref{eq:1}, \eqref{eq:2}, and \eqref{eq:3}, we indeed obtain \eqref{goal-ineq}.
\end{proof}

\subsection{Relation with geometric mean}

Our key observation is that, maps of the form $X \mapsto \Phi(f(X))$ satisfy a strong convexity
$$
\Phi(f(X \amn Y)) \le \Phi(f(X)) \hmn \Phi(f(Y))
$$
which leads to usual convexity under functional calculus by $g$.
This notably differs from the approach of~\cite{MR2805638} in that we do not use the \emph{geometric mean}
$$
X \gmn Y = X^{\hlf} \left(X^{-\hlf} Y X^{-\hlf}\right)^{\hlf} X^{\hlf}.
$$
Since this can be characterized as
$$
X \gmn Y = \max \biggl\{ Z \biggm| Z = Z^*, \left(\begin{array}{cc}X & Z \\Z & Y\end{array}\right) \ge 0 \biggr\}
$$
and positive maps satisfy a restricted form of $2$-positivity for block matrices of the above form, we do have $\Phi(X \gmn Y) \le \Phi(X) \gmn \Phi(Y)$ when $\Phi$ is a positive map, see also Remark~\ref{rem:conv-sym-mean}.

Combining this with the operator log-convexity $f(X \amn Y) \le f(X) \gmn f(Y)$, we obtain $\Phi(f(X \amn Y)) \le \Phi(f(X)) \gmn \Phi(f(Y))$.
(This can be also seen from the above observation as we always have $A \hmn B \le A \gmn B$ for positive $A$ and $B$.)

However, taking functional calculus by operator monotone functions is not compatible with taking harmonic mean in general.
For example,
\begin{align*}
S &= \left(\begin{array}{cc}1.1 & 0 \\0 & 0.1\end{array}\right),&
T &= \left(\begin{array}{cc}7.17 & -4.41 \\-4.41 & 3.13\end{array}\right)
\end{align*}
have the geometric mean
$$
S \gmn T = \left(\begin{array}{cc}1.85834\ldots & -0.63486\ldots \\-0.63486\ldots & 0.52569\ldots\end{array}\right).
$$
Then, for $g(x) = \sqrt x$, the matrix $\frac12(g(S) + g(T)) - g(S \gmn T)$ has eigenvalues
$$
\lambda_1 = 0.5786\ldots, \quad
\lambda_2 = -0.0159\ldots,
$$
so it seems difficult to derive the convexity of $X \mapsto g(\Phi(f(X)))$ using the geometric mean.

\section{Two variable version}

\label{sec:two-var-ver}

In~\citelist{\cite{MR3067823}\cite{MR3464069}} (see also \cite{MR3429039}), Hiai considered $2$-variate convexity / concavity problems involving positive maps.
Among his results is the following part of~\cite{MR3464069}*{Theorem 2.1}, which is closest to our setting: with $-1 \le p, q \le 0$, let $g(x)$ be a real function such that $g(x^{p+q})$ is operator monotone decreasing.
Then, for any strictly positive maps $\Phi\colon M_m(\bC) \to M_k(\bC)$ and $\Psi\colon M_n(\bC) \to M_k(\bC)$, the map
\[
%\begin{gather*}
M_m(\bC)^{++} \times M_n(\bC)^{++} \to \bR, \quad (X, Y) \mapsto \Tr\mathopen{}\left(g\mathopen{}\left(\Phi(X^p)^{\hlf}\Psi(Y^q)\Phi(X^p)^{\hlf}\right)\right)
%\end{gather*}
\]
is jointly convex.

Our result implies that, if $f_1(x)$ and $f_2(x)$ are operator monotone decreasing positive functions on $(0, \infty)$, and if $g(x)$ is as in Theorem \ref{thm:main}, then the map
\begin{equation}
\begin{gathered}
\label{eq:2-var-func}
M_m(\bC)^{++} \times M_n(\bC)^{++} \to \bR,\\
(X, Y) \mapsto \Tr\mathopen{}\left(g\mathopen{}\left(\Phi(f_1(X))^{\hlf}\Psi(f_2(Y))\Phi(f_1(X))^{\hlf}\right)\right)
\end{gathered}
\end{equation}
is \emph{separately} convex.
Indeed, if $X$ is fixed, the map
$$
M_n(\bC) \to M_k(\bC), \quad T \mapsto \Phi(f_1(X))^{\hlf}\Psi(T)\Phi(f_1(X))^{\hlf}
$$
is positive, hence we can apply Theorem~\ref{thm:main}.
Using the trace property, one can check the equality
\[
%\begin{multline*}
\Tr\mathopen{}\left(h\mathopen{}\left(\Phi(f_1(X))^{\hlf}\Psi(f_2(Y))\Phi(f_1(X))^{\hlf}\right)\right) = \Tr\mathopen{}\left(h\mathopen{}\left(\Psi(f_2(Y))^{\hlf}\Phi(f_1(X))\Psi(f_2(Y))^{\hlf}\right)\right)
%\end{multline*}
\]
when $h(x)$ is a polynomial function.
By uniform approximation on intervals we can replace $h$ by $g$, which allows us to switch the role of $X$ and $Y$.

While our method does not seem to have direct implication for joint convexity, for $g(x) = x$ the following variation of the argument of~\cite{MR0332080} shows that the map \eqref{eq:2-var-func} is indeed jointly convex.

\begin{proposition}
\label{prop:Lieb-conv}
Let $f_1(x)$ and $f_2(x)$ be operator monotone decreasing functions from $(0, \infty)$ to itself.
If $A$ and $B$ are unital C$^*$-algebras, $\tau$ is a tracial positive functional on $B$, $\Phi$ is a positive linear map $A \to B$, and $K \in B$, then the map
$$
A^{++} \to \bR, \quad X \mapsto \tau\mathopen{}\left(\Phi(f_1(X)) K^* \Phi(f_2(X)) K\right)
$$
is convex.
\end{proposition}

\begin{proof}
By linearity and the integral representation of the $f_i(x)$, we may assume $f_1(x) = 1/(\lambda + x)$ and $f_2(x) = 1/(\mu + x)$ for some $\lambda, \mu \ge 0$.
By small perturbation we may also assume that $K$ is invertible and $\Phi$ is strictly positive.
Let $X \in A^{++}$ and $Y = Y^* \in A$.
The claim follows if we prove that the function
$$
h(t) = \tau\mathopen{}\left(\Phi(f_1(X + t Y)) K^* \Phi(f_2(X + t Y)) K\right)
$$
defined for small $\absv{t}$ satisfies $\partial_t^2 h(t) \vert_{t=0} \ge 0$.

The derivative of $f_1(X + t Y)$ is given by
$$
\partial_t f_1(X + tY) = - (\lambda  + X + t Y)^{-1} Y (\lambda  + X + t Y)^{-1} = - f_1(X + t Y) Y f_1(X + t Y).
$$
We thus have
$$
\partial_t^2 f_i(X + tY) \vert_{t=0} = 2 f_i(X) Y f_i(X) Y f_i(X),
$$
for $i = 1, 2$, and
\begin{multline*}
(\partial_t^2 h)(0) = 2 \bigl(\tau(\Phi(f_1(X) Y f_1(X) Y f_1(X)) K^* \Phi(f_2(X)) K) \\
+ \tau(\Phi(f_1(X) Y f_1(X)) K^* \Phi(f_2(X) Y f_2(X)) K) + \tau(\Phi(f_1(X)) K^* \Phi(f_2(X) Y f_2(X) Y f_2(X)) K) \bigr).
\end{multline*}

Let us put
\begin{align*}
b_1 &= f_1(X)^{\hlf} Y f_1(X)^{\hlf},& b_2 &= f_2(X)^{\hlf} Y f_2(X)^{\hlf},\\ c_1 &= \Phi(f_1(X))^{\hlf}, & c_2 &= (K^* \Phi(f_2(X)) K)^{\hlf}
\end{align*}
and consider the maps
\begin{align*}
\Phi'(T) &= \Phi\mathopen{}\left(f_1(X)^{\hlf} T f_1(X)^{\hlf}\right),& \Phi'_u(T) &= c_1^{-1} \Phi'(T) c_1^{-1}\\
\Psi'(T) &= K^* \Phi\mathopen{}\left(f_2(X)^{\hlf} T f_2(X)^{\hlf}\right) K,& \Psi'_u(T) &= c_2^{-1} \Psi'(T) c_2^{-1}.
\end{align*}
(The elements $c_i$ are invertible by our additional assumptions on $K$ and $\Phi$.)
Thus, we want to prove
$$
2 \tau\mathopen{}\left(c_1 \Phi'_u(b_1^2) c_1 c_2^2 + c_1 \Phi'_u(b_1) c_1 c_2 \Psi'_u(b_2) c_2 + c_1^2 c_2 \Psi'_u(b_2^2) c_2 \right) \ge 0.
$$

Since $\Phi'_u$ and $\Psi'_u$ are unital positive maps, the tracial property of $\tau$ together with the (Jensen--)Kadison inequalities of the form $\Phi'_u(b_1^2) \ge \Phi'_u(b_1)^2$ imply
\begin{align*}
\tau\mathopen{}\left(c_1 \Phi'_u(b_1^2) c_1 c_2^2 \right) &\ge \tau(c_1 d_1^2 c_1 c_2^2),&
\tau\mathopen{}\left(c_1^2 c_2 \Psi'_u(b_2^2) c_2 \right) &\ge \tau(c_1^2 c_2 d_2^2 c_2)
\end{align*}
for $d_1 = \Phi'_u(b_1)$ and $d_2 = \Psi'_u(b_2)$.
Thus, it is enough to have
$$
2 \tau \mathopen{}\left(c_1 d_1^2 c_1 c_2^2 + c_1 d_1 c_1 c_2 d_2 c_2 + c_1^2 c_2 d_2^2 c_2 \right) \ge 0.
$$
Using the tracial property of $\tau$, one sees that the left hand is equal to
$$
\tau\mathopen{}\left((c_2 c_1 d_1 + d_2 c_2 c_1) (d_1 c_1 c_2 + c_1 c_2 d_2) + c_2 c_1 d_1^2 c_1 c_2 + c_1 c_2 d_2^2 c_2 c_1 \right),
$$
which is indeed nonnegative.
\end{proof}

\begin{theorem}
Let $f_1(x)$ and $f_2(x)$ be as in Proposition~\ref{prop:Lieb-conv}, and let $A_1$, $A_2$, $B$ be unital C$^*$-algebras, $\tau$ be a tracial positive functional on $B$, and $\Phi\colon A_1 \to B$ and $\Psi\colon A_2 \to B$ be positive linear maps.
Then the map
$$
A_1^{++} \times A_2^{++} \to \bR, \quad (X, Y) \mapsto \tau\mathopen{}\left(\Phi(f_1(X))^{\hlf}\Psi(f_2(Y))\Phi(f_1(X))^{\hlf}\right)
$$
is jointly convex.
\end{theorem}

\begin{proof}
Consider the map
$$
\tilde\Phi\colon A_1 \oplus A_2 \to M_2(B), \quad X \oplus Y \mapsto \left(\begin{array}{cc}\Phi(X) & 0 \\0 & \Psi(Y)\end{array}\right).
$$
This is a positive map, and for $(X, Y) \in A_1^{++} \times A_2^{++}$ the elements
\begin{align*}
Z &= X  \oplus Y \in (A_1 \oplus A_2)^{++},&
K &= \left(\begin{array}{cc}0 & 0 \\1 & 0\end{array}\right) \in M_{2}(B)
\end{align*}
satisfy
\[
%\begin{multline*}
(\tau\otimes\Tr)\mathopen{}\left(\tilde\Phi(f_1(Z)) K^* \tilde\Phi(f_2(Z)) K\right) = \tau\mathopen{}\left(\Phi(f_1(X))\Psi(f_2(Y))\right) = \tau\mathopen{}\left(\Phi(f_1(X))^{\hlf}\Psi(f_2(Y))\Phi(f_1(X))^{\hlf}\right)
%\end{multline*}
\]
up to the identification $M_2(B) \simeq B \otimes M_2(\bC)$.
Thus the assertion follows from Proposition \ref{prop:Lieb-conv}.
\end{proof}

\end{document}